\documentclass{amsart}

\usepackage{amsmath,amsthm}

\theoremstyle{plain}
\newtheorem{thm}{Theorem}[section]
\newtheorem{cor}[thm]{Corollary}

\newtheorem{lemma}[thm]{Lemma}

\theoremstyle{definition}

\newenvironment{thmthm}[2][Theorem]{\begin{trivlist}
\item[\hskip \labelsep {\bfseries #1}\hskip \labelsep {\bfseries #2}]}{\end{trivlist}}

\newcommand{\comment}[1]{}
\newcommand{\lk}{lk}
\newcommand{\bdry}{\ensuremath{\partial}}

\newcommand{\Q}{\ensuremath{\mathbb{Q}}}
\newcommand{\R}{\ensuremath{\mathbb{R}}}
\newcommand{\Z}{\ensuremath{\mathbb{Z}}}

\renewcommand{\sl}{\ensuremath{{\, sl}}}
\DeclareMathOperator{\tb}{tb}
\DeclareMathOperator{\rot}{rot}
\def\dfn#1{{\em #1}}

\title{Rational linking and contact geometry}

\author{Kenneth L.\ Baker}
\address{
    Department of Mathematics,
    University of Miami,
   PO Box 249085,
Coral Gables, FL 33124-4250}
\email{k.baker@math.miami.edu}
\urladdr{http://math.miami.edu/\char126 kenken}

\author{John B.\ Etnyre}
\address{
    School of Mathematics,
    Georgia Institute of Technology,
    686 Cherry St.,
    Atlanta, GA  30332-0160}
\email{etnyre@math.gatech.edu}
\urladdr{http://math.gatech.edu/\char126 etnyre}

\begin{document}
\maketitle
\begin{center}
{\em This paper is dedicated to Oleg Viro on the occasion of his 60th birthday.}
\end{center}

\begin{abstract}
In the note we study Legendrian and transverse knots in rationally null-homologous knot types. In particular we generalize the standard definitions of self-linking number, Thurston-Bennequin invariant and rotation number. We then prove a version of Bennequin's inequality for these knots and classify precisely when the Bennequin bound is sharp for fibered knot types. Finally we study rational unknots and show they are weakly Legendrian and transversely simple. 
\end{abstract}

In this note we extend the self-linking number of transverse knots and the Thurston-Bennequin invariant and rotation number of Legendrian knots to the case of rationally null-homologous knots. This allows us to generalize many of the classical theorems concerning Legendrian and transverse knots (such as the Bennequin inequality) as well as put other theorems in a more natural context (such as the result in \cite{EtnyreVanHornMorris08Pre} concerning exactness in the Bennequin bound). Moreover due to recent work on the Berge conjecture \cite{BakerGrigsbyHedden08} and surgery problems in general, it has become clear that one should consider rationally null-homologous knots even when studying classical questions about Dehn surgery on knots in $S^3.$   Indeed, the Thurston-Bennequin number of Legendrian rationally null-homolgous knots in lens spaces has been examined in \cite{BakerGrigsby08Pre}. There is also a version of the rational Thurston-Bennequin invariants for links in rational homology spheres that was perviously defined and studied in \cite{Oztruk05}.

We note that there has been work on relative versions of the self-linking number (and other classical invariants) to the case of general (even non null-homologus) knots, {\em cf} \cite{Chernov05}. While these relative invariants are interesting and useful, many of the results considered here do not have analogous statements. So rationally null-homologous knots seems to be one of the largest classes of knots to which one can generalize classical results in a straightforward manner. 

There is a well-known way to generalize the linking number between two null-homologous knots to rationally null-homologous knots, see for example \cite{GompfStipsicz99}. We recall this definition of a rational linking number in Section~\ref{sec:basics} and then proceed to define the rational self-liking number $sl_\Q(K)$ of a transverse knot $K$ and the rational Thurston-Bennequin invariant $\tb_\Q(L)$ and rational rotation number $\rot_\Q(L)$ of a Legendrian knot $L$ in a rationally null-homologous knot type. We also show the expected relation between these invariants of the transverse push-off of a Legendrian knot and of stabilizations of Legendrian and transverse knots. This leads to one of our main observations, a generalization of Bennequin's inequality.
 \begin{thmthm}{\ref{thm:ratBennequin}}{\em 
Let $(M,\xi)$ be a tight contact manifold  and suppose $K$ is a transverse knot in it of order $r>0$ in homology. Further suppose that $\Sigma$ is a rational Seifert surface of $K.$ Then
\[
sl_\Q(K)\leq -\frac 1r \chi(\Sigma).
\]
Moreover, if $K$ is Legendrian then
\[
\tb_\Q(K)+|\rot_\Q(K)|\leq -\frac 1r \chi(\Sigma).
\]
}\end{thmthm}
In \cite{EtnyreVanHornMorris08Pre}, bindings of open book decompositions that satisfied equality in the Bennequin inequality were classified. We generalize that result to the following. 
 \begin{thmthm}{\ref{thm:supportgen}}{\em 
Let $K$ be a rationally null-homologus, fibered transverse knot in a contact 3-manifold $(M,\xi)$ such that $\xi$ is tight when restricted to the complement of $K.$ Denote by $\Sigma$ a fiber in the fibration of $M-K$ and let $r$ be the order of $K.$ Then $r\, sl^\xi_\Q(K,\Sigma)=-\chi(\Sigma)$ if and only if  either $\xi$ agrees with the contact structure supported by the rational open book determined by $K$ or is obtained from it by adding Giroux torsion along tori which are incompressible in the complement of $L.$
}\end{thmthm}

A rational unknot in a manifold $M$ is a knot $K$ with a disk as a rational Seifert surface. One may easily check that if $M$ is irreducible then for $M$ to admit a rational unknot (that is not actually an unknot) it must be diffeomorphic to a lens space.
 \begin{thmthm}{\ref{thm:mainratknot}}{\em
Rational unknots in tight contact structures on lens spaces are weakly transversely simple and Legendrian simple.
}\end{thmthm}
In Section~\ref{sec:ratknot} we also given an example of the classification of Legendrian rational unknots (and hence transverse rational unknots) in $L(p,1)$ when $p$ is odd. The classification of Legendrian and transverse rational unknots in a general lens space can easily be worked out in terms of the classification of tight contact structures on the given lens space. The example we give illustrates this.

In Section~\ref{sec:links}, we briefly discuss the generalization of our results to the case of links.

{\bf Errata to published version.} In the published version of this paper the definition of the rational self-linking number was incorrect. It was missing a constant factor depending on the order of the knot. We are grateful to Chris Wendl who pointed this out to us. The paper is unchanged except for the definition of the rational self-linking number in Subsection~\ref{subsec:selflink} and a short discussion about the definition there. The remainder of the paper and all the proofs remain unchanged, indicating that we had the ``correct definition" in mind all along. Our original error occurred when relating (rational) linking numbers in a space and one of its covering spaces (specifically at the end of the proof of Lemma~\ref{computersl}).

{\em Acknowledgements.} The first author was partially supported by NSF Grant DMS-0239600. The second author was partially supported by NSF Grants  DMS-0239600 and DMS-0804820.

\section{Rational linking and  transverse and Legendrian knots}\label{sec:basics}
Let $K$ be an oriented knot of $\Z$--homological order $r>0$ in a $3$--manifold $M$ and denote a tubular neighborhood of it by $N(K)$. By $X(K)$ denote the knot exterior $\overline{M\setminus N(K)}.$ We fix a framing on $N(K).$ We know that half the $\Z$--homology of $\partial X(K)$ dies  when included into the $\Z$--homology of $X(K).$ Since $K$ has order $r$ it is easy to see there is an embedded $(r,s)$--curve on $\partial X(K)$ that bounds an oriented connected surface $\Sigma^\circ$ in $X(K).$ We can radially cone $\partial \Sigma^\circ \subset \partial X(K)=\partial N(K)$ in $N(K)$ to get a surface $\Sigma$ in $M$ whose interior is embedded in $M$ and whose boundary wraps $r$ times around $K.$ Such a surface $\Sigma$ will be called a \dfn{rational Seifert surface} for $K$ and we say that $K$ \dfn{$r$--bounds} $\Sigma.$  We also sometime say $\Sigma$ is \dfn{order} $r$ along $K.$ We also call $\Sigma\cap \partial N(K)$ the \dfn{Seifert cable of $K$}.
Notice that  $\Sigma$ may have more than one boundary component. Specifically, $\Sigma$ will have $\gcd(r,s)$ boundary components. We call the number of boundary components of $\Sigma$ the \dfn{multiplicity} of $K.$
Notice $\Sigma$ defines an $\Z$--homology chain $\Sigma$ and $\partial \Sigma= rK$ in the homology 1-chains. In particular, as $\Q$--homology chains $\partial (\frac 1r  \Sigma)=K.$

We now define the \dfn{rational linking number} of another oriented knot $K'$ with $K$ (and Seifert surface $\Sigma$) to be 
\[
lk_\Q(K,K')=\frac 1r \Sigma\cdot K',
\]
where $\cdot$ denotes the algebraic intersection of $\Sigma$ and $K'.$ It is not hard to check that $lk_\Q$ is well-defined given the choice of $[\Sigma] \in H_2(X(K),\bdry X(K))$.  Choosing another rational Seifert surface for $K$ representing a different relative $2$nd homology class in $X(K)$ may change this rational linking number by a multiple of $\frac 1r$.  To emphasize this, one may prefer to write $lk_\Q((K,[\Sigma]),K')$.   Notice that if there exist rational Seifert surfaces $\Sigma_1$ and $\Sigma_2$ for which $lk_\Q((K,[\Sigma_1]),K') \neq lk_\Q((K,[\Sigma_2]),K')$, then $K'$ is not rationally null-homologous.


Moreover, if $K'$ is also rationally null-homologous then it $r'$--bounds a rational Seifert surface $\Sigma'.$ In $M\times[0,1]$ with $\Sigma$ and $\Sigma'$ thought of as subsets of $M\times\{1\}$ we can perturb them relative to the boundary to make them transverse. Then one may also check that
\[
lk_\Q(K,K')=\frac 1{rr'} \Sigma\cdot \Sigma'.
\]
From this one readily sees that the rational linking number of rationally null-homologous links is symmetric. 

\subsection{Transverse knots}\label{subsec:selflink}
Let  $(M,\xi)$ be a contact 3--manifold (with orientable contact structure $\xi$) and $K$ a (positively) transverse knot. Given a rational Seifert surface $\Sigma$ for $K$ with $\partial \Sigma=rK$ then we can trivialize $\xi$ along $\Sigma.$ More precisely we can trivialize the pull-back $i^*\xi$ to $\Sigma$ where $i:\Sigma\to M$ is the inclusion map. Notice that the inclusion map restricted to $\partial \Sigma$ is an $r$--fold covering map of $\partial \Sigma$ to $K.$ We can use the exponential map to identify a neighborhood of the zero section of $i^*\xi|_{\partial\Sigma}$ with an $r$--fold cover of a tubular neighborhood of $K.$ Let $v$ be a non-zero section of $i^*\xi.$ 
By choosing $v$ generically and suitably small the image of $v|_{\partial \Sigma}$ gives an embedded, oriented knot (or link) $K'$ in a neighborhood of $K$ that is disjoint from $K.$ We define the \dfn{rational self-linking number}\footnote{In the published version of this paper we defined the rational self-linking number to be $\lk(K,K')$. This turns out to be the wrong definition if one wants to generalize results from the null-homologous to the rationally null-homologous case.} to be
\[
sl_\Q(K)=\frac 1r \lk_\Q(K,K') = \frac 1{r^2} \Sigma\cdot K'.
\]
Notice the extra $\frac 1r$ in the definition. This is there because $K'$ is an $r$--fold push-off of $K$ not just a push-off. Notice that $\lk_\Q(K,K')$ is always an integer, or equivalently $\Sigma\cdot K'$ is a multiple of $r$. To see this let $N$ be a small neighborhood of $K$ in $M$ such that $\Sigma$ intersects $\partial N$ transversely. We can assume that $K'$ sits in $\partial N$ as an $(r,s)$ curve (since $K'$ has order $r$ in $N$) and similarly  $S=\Sigma\cap \partial N$ is an $(r,s')$ curve. Thus we see that $\Sigma\cdot K'=S\cdot K'=rs'-rs=r(s-s')$. From this it is easy to see that there is a vector field $v'$ along $K$ that pulls back under the covering map discussed above to the vector field $v$. Using $v'$ we can get a push-off $K''$ of $K$ and then define $\sl_\Q(K)=\lk(K,K'')$. We prefer the definition above since it is not {\em a priori} clear that $v'$ exists. 

It is standard to check that $sl_\Q$ is independent of the trivialization of $i^*\xi$ and the section $v.$ Moreover, the rational self-linking number only depends on the relative homology class of $\Sigma.$ When this dependence is important to note we denote the rational self-linking number as 
\[
sl_\Q(K,[\Sigma]).
\] 

Just as in the case of the self-linking number one can compute it by considering the characteristic foliation on $\Sigma.$ To this end we can always isotop $\Sigma$ so that its characteristic foliation $\Sigma_\xi$ is generic (in particular has only elliptic and hyperbolic singularities)  and we denote by $e_\pm$ the number of $\pm$--elliptic singular points and similarly $h_\pm$ denotes the number of $\pm$--hyperbolic points. 
\begin{lemma}\label{computersl}
Suppose $K$ is a transverse knot in a contact manifold $(M,\xi)$ that $r$--bounds the rational Seifert surface $\Sigma.$ Then
\begin{equation}\label{eq:comp}
sl_\Q(K,[\Sigma])=\frac 1r \left((e_--h_-)-(e_+-h_+)\right).
\end{equation}
\end{lemma}
\begin{proof}
We will consider the case when $\Sigma$ has a single boundary component and leave the other case to the reader. 
We begin by constructing a nice neighborhood of $\Sigma$ in $(M,\xi).$ To this end notice that for suitably small $\epsilon,$  $K$ has a neighborhood $N$ that is contactomorphic to the image $C_\epsilon$ of $\{(r,\theta, z) : r\leq \epsilon\}$ in $(\R^3,\ker (dz+r^2\, d\theta))$ modulo the action $z\mapsto z+1.$ Let $C'$ be the $r$--fold cover of $C_\epsilon.$ Taking $\epsilon$ sufficiently small we can assume that $\Sigma\cap \partial N$ is a transverse curve $T.$ Thinking of $T$ as sitting in $C_\epsilon$ we can take its lift $T'$ to $C'.$ Let $N'$ be a small neighborhood of $\overline{\Sigma- (N\cap \Sigma)}.$ We can glue $N'$ to $C_\epsilon$ along a neighborhood of $T$ to get a model neighborhood $U$ for $\Sigma$ in $M.$ Moreover we can glue $N'$ to $C'$ along a neighborhood of $T'$ to get a contact manifold $U'$ that will map onto $U$ so that $C'$ $r$--fold covers $C_\epsilon$ and $N'$ in $U'$ maps diffeomorphically to $N'$ in $U.$ Inside $U'$ we have $K' = \partial \Sigma$ which $r$--fold covers $K$ in $U.$ The transverse knot $K'$ is a null-homolgous knot in $U'$. According to a well known formula that easily follows by interpreting $sl(K')$ as a relative Euler class, see \cite{Eliashberg90b}, we have that 
\[
sl(K')=(e_--h_-)-(e_+-h_+),
\]
where $e_\pm$ and $h_\pm$ are as in the statement of the theorem. Now one easily sees that $sl_\Q(K)=\frac 1r sl(K')$ from which the lemma follows. 
\end{proof}

\subsection{Legendrian knots}
Let  $(M,\xi)$ be a contact 3--manifold (with orientable contact structure $\xi$) and $K$ a Legendrian knot. Choose a framing on $K.$ Given a rational Seifert surface $\Sigma$ for $K$ the Seifert cable of $K$ is $K_{(r,s)}.$

The restriction $\xi\vert_K$ induces a framing on the normal bundle of $K$.  Define the \dfn{(rational) Thurston-Bennequin number} of the Legendrian knot $K$ to be 
\[ \tb_\Q(K) = lk_\Q (K,K'), \]
where $K'$ is a copy of $K$ obtained by pushing off using the framing coming from $\xi.$

We now assume that $K$ is oriented.  
Recall the inclusion $i \colon \Sigma \hookrightarrow M$ is an embedding on the interior of $\Sigma$ and an $r$--to--$1$ cover $\bdry \Sigma \to K$.  
As above we can trivialize $\xi$ along $\Sigma.$ That is we can trivialize the pull-back $i^*\xi$ to $\Sigma.$
The oriented tangent vectors $T_K$ give a section of $\xi\vert_K$.    Thus $i^*T_K$ gives a section of $\R^2 \times \bdry \Sigma$.  Define the \dfn{rational rotation number} of the Legendrian knot $K$ to be the winding number of $i^*T_K$ in $\R^2$ divided by $r$
\[\rot_\Q(K)=\frac 1 r \text{winding}(i^*T_K, \R^2).\]

Recall \cite{Etnyre05} that given a Legendrian knot $K$ we can always form the \dfn{(positive) transverse push-off of $K$}, denoted $T(K),$ as follows: the knot $K$ has a neighborhood contactomorphic to the image of the $x$--axis in $(\R^3, \ker (dz-y\,dx))$ modulo the action $x\mapsto x+1$ so that the orientation on the knot points towards increasing $x$--values. The curve $\{(x, \epsilon, 0)\}$ for $\epsilon>0$ small enough will give the transverse push-off of $K.$

\begin{lemma}\label{lem:pushoff}
If $K$ is a rationally null-homologous Legendrian knot in a contact manifold $(M,\xi)$ then
\[
sl_\Q(T(K))=\tb_\Q(K)-\rot_\Q(K).
\]
\end{lemma}
\begin{proof}
Notice that pulling $K$ back to a cover $U'$ similar to the one constructed in the proof of Lemma~\ref{computersl} we get a null-homologous Legendrian knot $K'.$ Here we have the well-known formula, see  \cite{Etnyre05},
\[
sl(T(K'))=\tb(K')-\rot(K').
\]
One easily computes that $r\, sl(T(K'))=sl_\Q(T(K)), r\, \tb(K')=\tb_\Q(K')$ and $r\, \rot(K')=\rot(K).$ The lemma follows. 
\end{proof}

We can also construct a Legendrian knot from a transverse knot. Given a transverse knot $K$ it has a neighborhood as constructed in the proof of Lemma~\ref{computersl}. It is clear that the boundary of  small enough closed neighborhood of $K$ of the appropriate size will have a linear characteristic foliation by longitudes of $K$. One of the leaves in this characteristic foliation will be called a \dfn{Legendrian push-off of $K$}. We note that this push-off is not unique, but that different Legendrian push-offs are related by negative stabilizations, see \cite{EtnyreHonda01b}.

\subsection{Stabilization}
Recall that stabilization of a transverse and Legendrian knot is a local procedure near a point on the knot so it can be performed on any transverse or Legendrian knot whether null-homologous or not. 

There are two types of stabilization of a Legendrian knot $K$, positive and negative stabilization, denoted $S_+(K)$ and $S_-(K),$ respectively. Recall, that if one identifies a neighborhood of a point on a Legendrian knot with a neighborhood of the origin in $(\R^3, \ker(dz-y\, dx))$ so that the Legendrian knot is mapped to a segment of the $x$--axis and the orientation induced on the $x$--axis from $K$ is points towards increasing $x$--values then $S_+(K),$ respectively $S_-(K),$ is obtained by replacing the segment of the $x$--axis by a ``downward zigzag'', respectively ``upward zigzag'', see \cite[Figure 19]{Etnyre05}. One may similarly define stabilization of a transverse knot $K$ and we denote it by $S(K).$ Stabilizations have the same effect on the rationally null-homologous knots as they have on null-homologous ones. 
\begin{lemma}\label{lem:stabilize}
Let $K$ be a rationally null-homolgous Legendrian knot in a contact manifold. Then
\[
\tb_\Q(S_\pm(K))=\tb_\Q(K)-1 \text{ and } \rot_\Q(S_\pm(K))=\rot_\Q(K)\pm1. 
\]
Let $K$ be a rationally null-homologous transverse knot in a contact manifold. Then 
\[
sl_\Q(S(K))=sl_\Q(K)-2.
\]
\end{lemma}
\begin{proof}
One may check that if $K'$ is a push off of $K$ by some framing $\mathcal{F}$ and $K''$ is the push off of $K$ by a framing $\mathcal{F}''$ such that the difference between $\mathcal{F}$ and $\mathcal{F}'$  is $-1$ then 
\[
lk_\Q(K,K'')=lk_\Q(K,K')-1.
\]
Indeed by noting that $r\, lk_\Q(K,K')$ can easily be computed by intersecting the Seifert cable of $K$ on the boundary of a neighborhood of $K,$ $T^2=\partial N(K),$ with the curve $K'\subset T^2,$ the result easily follows. From this one obtains the change in $\tb_\Q.$

Given a rational Seifert surface $\Sigma$ that is $r$--bounded by $K,$ a small Darboux neighborhood $N$ of a point $p\in K$ intersects $\Sigma$ in $r$ disjoint disks. Since the stabilization can be performed in $N$ it is easy to see $\Sigma$ is altered by adding $r$ small disks, each containing a positive elliptic point and negative hyperbolic point (see \cite{Etnyre05}). The result for $sl_\Q$ follows. 

Finally the result for $\rot_\Q$ follows by a similar argument or from the previous two results, Lemma~\ref{lem:pushoff} and the next lemma (whose proof does not explicitly use the rotation number results from this lemma). 
\end{proof}

The proof of the following lemma is given in  \cite{EtnyreHonda01b}.
\begin{lemma}\label{lem:stabLeg}
Two transverse knots in a contact manifold are transversely isotopic if and only if they have Legendrian push-offs that are Legendrian isotopic after each has been negatively stabilized some number of times. The same statement is true with ``transversely isotopic'' and ``Legendrian isotopic'' both replaced by ``contactomorphic''.\qed
\end{lemma}
We similarly have the following result.
\begin{lemma}
Two Legendrian knots representing the same topological knot type are Legendrian isotopic after each has been positively and negatively stabilized some number of times. \qed
\end{lemma}
While this is an interesting result in its own right it clarifies the range of possible values for $\tb_\Q.$ More precisely the following result is an immediate corollary.
\begin{cor}
If two Legendrian knots represent the same topological knot type then the difference in their rational Thurston-Bennequin invariants is an integer. 
\end{cor} 

\section{The Bennequin bound}\label{sec:bennequin}
Recall that in a tight contact structure the self-linking number of a null-homologous knot $K$ satisfies the well-known Bennequin bound
\[
sl(K)\leq -\chi(\Sigma)
\]
for any Seifert surface $\Sigma$ for $K,$ see \cite{Eliashberg93}. We have the analogous result for rationally null-homologous knots.
\begin{thm}\label{thm:ratBennequin}
Suppose $K$ is a transverse knot in a tight contact manifold $(M,\xi)$ that $r$--bounds the rational Seifert surface $\Sigma.$ Then
\begin{equation}\label{eq:bie}
sl_\Q(K,[\Sigma])\leq - \frac 1r \chi(\Sigma).
\end{equation}
If $K$ is a Legendrian knot then
\[
\tb_\Q(K)+|\rot_\Q(K)|\leq -\frac 1r \chi(\Sigma).
\]
\end{thm}
\begin{proof}
The proof is essentially the same as the one given in \cite{Eliashberg93}, see also \cite{Etnyre03}. The first thing we observe is that if $v$ is a vector field that directs $\Sigma_\xi,$ that is $v$ is zero only at the singularities of $\Sigma_\xi$ and points in the direction of the orientation of the non-singular leaves of $\Sigma_\xi,$ then $v$ is a generic section of the tangent bundle of $\Sigma$ and points out of $\Sigma$ along $\partial\Sigma.$ Thus the Poincar\'e-Hopf theorem implies
\[
\chi(\Sigma)=(e_+ - h_+)+(e_- - h_-).
\] 
Adding this equality to $r$ times Equation~\eqref{eq:comp} gives
\[
r\, sl_\Q(K,[\Sigma])+\chi(\Sigma)= 2(e_- - h_-).
\]
So if we can isotop $\Sigma$ relative to the boundary so that $e_-=0$ then we clearly have the desired inequality. Recall that if an elliptic point and a hyperbolic point of the same sign are connected by a leaf in the characteristic foliation then they may be cancelled (without introducing any further singular points). Thus we are left to show that for every negative elliptic point we can find a negative hyperbolic point that cancels it. To this end, given a negative elliptic point $p$ consider the \dfn{basin} of $p,$ that is the closure of the set of points in $\Sigma$ that limit under the flow of $v$ in backwards time to $p.$ Denote this set by $B_p.$ Since the flow of $v$ goes out the boundary of $\Sigma$ it is clear that $B_p$ is contained in the interior of $\Sigma.$ Thus we may analyze $B_p$ exactly as in \cite{Eliashberg93, Etnyre03} to find the desired negative hyperbolic point. We briefly recall the main points of this argument.  First, if there are repelling periodic orbits in the characteristic foliation then add canceling pairs of positive elliptic and hyperbolic singularities to eliminate them. This prevents any periodic orbits in $B_p$ and thus one can show that $B_p$ is the immersed image of a polygon that is an embedding on its interior.  If $B_p$ is the image of an embedding then the boundary consists of positive elliptic singularities and hyperbolic singularities of either sign and flow lines between these singularities. If one of the hyperbolic singularities is negative then we are done as it is connected to $B_p$ by a flow line. If none of the hyperbolic points are negative then we can cancel them all with the positive elliptic singularities in $\partial B_p$ so that $\partial B_p$ becomes a periodic orbit in the characteristic foliation and, more to the point, the boundary of an overtwisted disk. In the case where $B_p$ is an immersed polygon one may argue similarly, see \cite{Eliashberg93, Etnyre03}.

The inequality for Legendrian $K$ clearly follows from considering the positive transverse push-off of $K$ and $-K$ and Lemma~\ref{lem:pushoff} together with the inequality in the transverse case. 
\end{proof}

\section{Rational open book decompositions and cabling}\label{sec:obcable}

A \dfn{rational open book decomposition} for a manifold $M$ is a pair $(L, \pi)$ consisting of 
\begin{itemize}
\item an oriented link $L$ in $M$ and 
\item a fibration $\pi \colon (M\setminus L)\to S^1$
\end{itemize}
such that no component of $\pi^{-1}(\theta)$  meets a component of $L$ meridionally for any $\theta\in S^1$.  We note that $\pi^{-1}(\theta)$ is a rational Seifert surface for the link $L.$ If $\pi^{-1}(\theta)$ is actually a Seifert surface for $L$ then we say that $(L,\pi)$ is an \dfn{open book decomposition} of $M$ (or sometimes we will say an \dfn{integral} or \dfn{honest} open book decomposition for $M$).  We call $L$ the \dfn{binding} of the open book decomposition and $\overline{\pi^{-1}(\theta)}$ a \dfn{page}.

The rational open book decomposition $(L,\pi)$ for $M$ \dfn{supports} a contact structure $\xi$ if there is a contact form $\alpha$ for $\xi$ such that 
\begin{itemize}
\item $\alpha(v)>0$ for all positively pointing tangent vectors $v\in TL,$ and
\item $d\alpha$ is a volume form when restricted to (the interior of) each page of the open book.
\end{itemize}
Generalizing the work of Thurston and Winkelnkemper \cite{ThurstonWinkelnkemper75}, the authors in work with Van Horn-Morris showed the following result.
\begin{thm}[Baker, Etnyre, Van Horn-Morris, 2008 \cite{BakerEtnyreVanHornMorris08Pre}]
Let $(L,\pi)$ be any rational open books decomposition of $M.$ Then there exists a unique contact structure $\xi_{(L,\pi)}$ that is supported by $(L,\pi).$
\end{thm}

It is frequently useful to deal with only honest open book decompositions. One may easily pass from a rational open book decomposition to an honest one using cables as we now demonstrate. 

Given any knot $K$,  let $N(K)$ be a tubular neighborhood of $K$, choose an orientation on $K$, an oriented meridian $\mu$ linking $K$ positively once, and choose some oriented framing (i.e.\ longitude) $\lambda$ on $K$ so that $\{\lambda, \mu\}$ give longitude-meridian coordinates on $\partial N(K).$  The \dfn{$(p,q)$--cable} of $K$ is the embedded curve (or collection of curves if $p$ and $q$ are not relatively prime) on $\partial N(K)$ in the homology class $p\lambda+q\mu.$ Denote this curve (these curves) by $K_{p,q}$.   We say a cabling of $K$ is \dfn{positive} if the cabling coefficients have slope greater than the Seifert slope of $K$.  (The \dfn{slope} of the homology class $p\lambda+q\mu$ is $q/p$.)

If $K$ is also a transverse knot with respect to a contact structure on $M$, then using the contactomorphism in the proof of Lemma~\ref{computersl}  between the neighborhood $N=N(K)$ and $C_\epsilon$ for sufficiently small $\epsilon$ we may assume that the cable $K_{p,q}$ on $\partial N$ is also transverse.  As such, we call $K_{p,q}$ the \dfn{transverse $(p,q)$--cable}.

 If $L=K_1\cup\dots\cup K_n$ is a link then we can fix framings on each component of $L$ and choose $n$ pairs of integers $(p_i,q_i),$ then after setting $({\bf p},{\bf q})=((p_1,q_1),\ldots,(p_n,q_n))$ we denote by $L_{({\bf p},{\bf q})}$ the result of $(p_i,q_i)$--cabling $K_i$ for each $i.$  It is easy to check, see for example  \cite{BakerEtnyreVanHornMorris08Pre}, that if $L$ is the binding of a rational open book decomposition of $M$ then so is $L_{({\bf p},{\bf q})}$ unless a component $K_i$ of $L$ is nontrivially cabled by curves of the fibration's restriction to $\partial N(K_i)$.

The following lemma says how the Euler characteristic of the fiber changes under cabling as well as the multiplicity and order of a knot. 

\begin{lemma}\label{lem:cable}
Let $L$ be a (rationally null-homologous) fibered link in $M.$ Suppose $K$ is a component of $L$ for which the fibers in the fibration approach as $(r,s)$--curves (in some framing on $K$). Let $L'$ be the link formed from $L$ by replacing $K$ by the $(p,q)$--cable of $K$ where $p\neq \pm1,0$ and $(p,q)\neq (kr,ks)$ for any $k \in \Q$.  Then $L'$ is fibered.  Moreover the Euler characteristic of the new fiber is
\[
\chi(\Sigma_\text{new})= \frac{1}{\gcd(p,r)} \left( |p| \chi(\Sigma_\text{old}) + |{ps-qr}|(1-|p|)\right),
\]
where $\Sigma_\text{new}$ is the fiber of $L'$ and $\Sigma_\text{old}$ is the fiber of $L.$ The multiplicity of each component of the cable of $K$ is 
\[
\gcd \left( \frac{r}{\gcd(p,r)}, \frac{p(rq-sp)}{\gcd(p,r)\gcd(p,q)} \right)
\]
and the order of $\Sigma_\text{new}$ along each component of the cable of $K$ is
\[
\frac{r}{\gcd(p,r)}.
\]\qed
\end{lemma}

The proof of this lemma may be found in \cite{BakerEtnyreVanHornMorris08Pre} but easily follows by observing one may construct $\Sigma_\text{new}$  by taking $|\frac{p}{\gcd(p,r)}|$ copies of $\Sigma_\text{old}$ and $|\frac {rq-sp}{\gcd(p,r)}|$ copies of meridional disks to $K$ and connecting them via $|\frac{p(rq-sp)}{\gcd(p,r)}|$  half twisted bands.  

Now suppose we are given a rational open book decomposition $(L,\pi)$ of $M.$  Suppose $K$ is a rational binding component of an open book $(L,\pi)$ whose page approaches $K$ in a $(r,s)$--curve with respect to some framing on $K$.  (Note $r \neq 1$ and is not necessarily coprime to $s$.)  For any $l \neq s$ replacing $K$ in $L$ by the $(r,l)$--cable of $K$ gives a new link $L_{K_{(r,l)}}$ that by Lemma~\ref{lem:cable} is the binding of an (possibly rational) open book for $M$ and has $\gcd(r,l)$ new components each having order and multiplicity 1. This is called the \dfn{$(r,l)$--resolution of $L$ along $K.$}  In the resolution, the new fiber is created using just one copy of the old fiber following the construction of the previous paragraph. Thus after resolving $L$ along the other rational binding components, we have a new fibered link $L'$ that is the binding of an integral open book $(L', \pi')$.  This is called an \dfn{integral resolution of $L$}.  
If we always choose the cabling coefficients $(r,l)$ to have slope greater than the original coefficients $(r,s)$ then we say that we have constructed a \dfn{positive (integral) resolution of $L$}.

\begin{thm}[Baker, Etnyre and Van Horn-Morris, 2008 \cite{BakerEtnyreVanHornMorris08Pre}]\label{thm:resolve}
Let $(L,\pi)$ be a rational open book for $M$ supporting the contact structure $\xi.$ If $L'$ is a positive resolution of $L,$ then $L'$ is the binding of an integral open book decomposition for $M$ that also supports $\xi.$
\end{thm}

\section{Fibered knots and the Bennequin bound}\label{sec:fiberedbenn}
Recall that in \cite{EtnyreVanHornMorris08Pre} null-homologous (nicely) fibered links satisfying the Bennequin bound were classified. In particular, the following theorem was proven.
\begin{thm}[Etnyre and Van Horn-Morris, 2008 \cite{EtnyreVanHornMorris08Pre}]\label{thm:support}
  \label{thm:G-compat}
Let $L$ be a fibered transverse link in a contact 3-manifold $(M, \xi)$ and assume that $\xi$ is tight when restricted to $M\setminus L$. Moreover assume $L$ is the binding of an (integral) open book decomposition of $M$ with page $\Sigma.$  Then $sl_\xi{(L,\Sigma)} = -\chi(\Sigma)$ if and only if either
\begin{enumerate}
\item $\xi$ is supported by $(L, \Sigma)$ or
\item $\xi$ is obtained from $\xi_{(L,\Sigma)}$ by adding Giroux torsion along tori which are incompressible in the complement of $L$.  
\end{enumerate}
\end{thm}
In this section we generalize this theorem to allow for any rationally null-homologous knots, the link case will be dealt with in Section~\ref{sec:links}.

\begin{thm}\label{thm:supportgen}  
Let $K$ be a rationally null-homologus, fibered transverse knot in a contact $3$--manifold $(M,\xi)$ such that $\xi$ is tight when restricted to the complement of $K.$ Denote by $\Sigma$ a fiber in the fibration of $M-K$ and let $r$ be the order of $K.$ Then $r\, sl^\xi_\Q(K,\Sigma)=-\chi(\Sigma)$  if and only if  either 
\begin{enumerate}
\item $\xi$ agrees with the contact structure supported by the rational open book determined by $K$ or 
\item is obtained from it by adding Giroux torsion along tori which are incompressible in the complement of $K.$
\end{enumerate}
\end{thm}
\begin{proof}
Let $K'$ be a positive integral resolution of $K.$ Then from Theorem~\ref{thm:resolve} we know that $K'$ and $K$ support the same contact structure. In addition the following lemma (with Lemma~\ref{lem:cable}) implies that if  $r\, sl^\xi_\Q(K,\Sigma)=-\chi(\Sigma)$ then $sl_\xi(K',\Sigma')=-\chi(\Sigma')$ where $\Sigma'$ is a fiber in the fibration of $M-K'.$
Thus the proof is finished by Theorem~\ref{thm:support}. The other implication is obvious from the generalization of the Thurston-Winkelnkemper construction in \cite{BakerEtnyreVanHornMorris08Pre} and Equation~\eqref{eq:comp}, since the characteristic foliation on the page of a rational open book only contains positive singularities and while adding Giroux torsion adds negative singularities they cancel in the computation of the rational self-linking number and the Euler characteristic of $\Sigma.$ 
\end{proof}

\begin{lemma}\label{lem:slcable} 
Let $K$ be a rationally null-homologous transverse knot of order $r$ in a contact $3$--manifold. Fix some framing on $K$ and suppose a rational Seifert surface $\Sigma$ approaches $K$ as a cone on a $(r,s)$--knot. Let $K'$ be a $(p,q)$--cable of $K$ that is positive and transverse in the sense described before Lemma~\ref{lem:cable} and let $\Sigma'$ be the Seifert surface for $K'$ constructed from $\Sigma$ as in the previous section then  
\[
sl(K', [\Sigma'])= \frac{1}{\gcd(r,p)} \left( |p| r\, sl_\Q(K, [\Sigma]) + |{rq-sp}|(-1+|p|)\right).
\]
\end{lemma}
\begin{proof}
For each singular point in the characteristic foliation of $\Sigma$ there are $\frac{|p|}{\gcd(p,r)}$ corresponding singular points on $\Sigma'$ (coming from the  $\frac{|p|}{\gcd(p,r)}$ copies of $\Sigma$ used in the construction of $\Sigma'$). For each of the $\frac{|rq-sp|}{\gcd(p,r)}$ meridional disk used to construct $\Sigma'$ we get one positive elliptic point in the characteristic foliation of $\Sigma'.$ Finally, since cabling was positive the $|\frac{p(rq-sp)}{\gcd(p,r)}|$  half twisted bands added to create $\Sigma'$ each have a single positive hyperbolic singularity in their characteristic foliation. (It is easy to check the characteristic foliation is as described as the construction mainly takes place in a solid torus neighborhood of $K$ where we can write an explicit model for this construction.) The lemma follows now from Lemma~\ref{computersl}.
\end{proof}

\section{Rational unknots}\label{sec:ratknot}
A knot $K$ in manifold $M$ is called a \dfn{rational unknot} if a rational Seifert surface $D$ for $K$ is a disk. Notice that a neighborhood of $K$ union a neighborhood of $D$ is a punctured lens space. Thus the only manifold to have rational unknots (that are not actual unknots) are manifolds with a lens space summand. In particular the only irreducible manifolds with rational unknots (that are not actual unknots) are lens spaces. So we restrict our attention to lens spaces in this section. 
A knot $K$ in a lens space is a rational unknot if and only if the complement of a tubular neighborhood of $K$ is diffeomorphic to a solid torus. This of course implies that the rational unknots in $L(p,q)$ are precisely the cores of the Heegaard tori. 
\begin{thm}\label{thm:mainratknot}
Rational unknots in tight contact structures on lens spaces are weakly transversely simple and Legendrian simple.
\end{thm}
A knot type is \dfn{weakly transversely simple} if it is determined up to contactomorphism (topologically) isotopic to the identity by its knot type and (rational) self-lining number. We have the analogous definition for \dfn{weakly Legendrian simple}. We will prove this theorem in the standard way. That is we identify the maximal value for the rational Thurston-Bennequain invariant, show that there is a unique Legendrian knot with that rational Thurston-Bennequin invariant and finally show that any transverse unknot with non-maximal rational Thurston-Bennequin invariant can be destabilized. The transverse result follows from the Legendrian result as Lemma~\ref{lem:stabLeg} shows.

\subsection{Topological rational unknots}

We explicitly describe $L(p,q)$ as follows: fix $p>q>0$ and set
\[
L(p,q)=V_0\cup_\phi V_1
\]
where $V_i=S^1\times D^2$ and we are thinking of $S^1$ and $D^2$ as the unit complex circle and disk, respectively. In addition the gluing map $\phi:\partial V_{1}\to \partial V_{0}$ is given in standard 
longitude-meridian coordinates on the torus  by the matrix
$$\begin{pmatrix} -p' & p\\ q' & -q \end{pmatrix},$$
where $p'$ and $q'$ satisfy $pq'-p'q=-1$ and $p>p'>0, q\geq q'>0.$ We can find such $p', q'$ by taking a continued fraction expansion of $-\frac pq$
\[
-\frac pq =a_0-\frac{1}{a_1-\ldots\frac{1}{a_{k-1}-\frac{1}{a_k}}}
\]
with each $a_i\geq 2$ and then defining 
\[
-\frac {p'}{q'}=a_0-\frac{1}{a_1-\ldots\frac{1}{a_{k-1}-\frac{1}{a_k+1}}}.
\]

Since we have seen that a rational unknot must be isotopic to the core of a Heegaard torus we clearly have four possible (oriented) rational unknots: namely $K_0, -K_0, K_1$ and $-K_1$ where $K_i=S^1\times \{ pt\}\subset V_i.$ We notice that $K_0$ represents a generator in the homology of $L(p,q)$ and $-K_0$ is the negative of that generator. So except in $L(2,1)$ the knots $K_0$ and $-K_0$ are not isotopic or homeomorphic via a homeomorphism isotopic to the identity. Similarly for $K_1$ and $-K_1.$ Moreover, in homology $q[K_0]=[K_1].$ So if $q\neq1$ or $p-1$ then $K_1$ is not homeomorphic via a homeomorphism isotopic to the identity to $K_0$ or $-K_0.$ We have established most of the following lemma.
\begin{lemma}
The set of rational unknots up to homeomorphism isotopic to the identity in $L(p,q)$ is given by
\[
\{\text{rational unknots in $L(p,q)$}\}=
\begin{cases}
\{K_1\}& p=2\\
\{K_1, -K_1\}& p\not=2, q=1 \text{ or } p-1\\
\{K_0,-K_0, K_1, -K_1\}& q\not= 1 \text{ or } p-1
\end{cases}
\]
\end{lemma}
\begin{proof}
Recall that $L(p,q)$ is an $S^1$--bundle over $S^2$ if and only if $q=1$ or $p-1.$ In this case $K_0$ and $K_1$ are both fibers in this fibration and hence are isotopic. We are left to see that $K_0$ and $-K_0$ are isotopic in $L(2,1)=\R P^3.$ To this end notice that $K_0$ can be thought of as an $\R P^1.$ In addition, we have the natural inclusions $\R P^1\subset \R P^2\subset \R P^3.$ It is easy to find an isotopy of $\R P^1=K_0$ in $\R P^2$ that reverse the orientation. This isotopy easily extends to $\R P^3.$
\end{proof}

\subsection{Legendrian rational unknots}

For results concerning convex surfaces and standard neighborhoods of Legendrian knots we refer the reader to \cite{EtnyreHonda01b}.

Recall, in the classification of tight contact structures on $L(p,q)$ given in  \cite{Honda00a} the following lemma was proven as part of Proposition 4.17.
\begin{lemma}\label{startclass}
Let $N$ be a standard neighborhood of a Legendrian knot isotopic to the rational unknot $K_1$ in a tight contact structure on $L(p,q).$ Then there is another neighborhood with convex boundary $N'$ such that $N\subset N'$  and $\partial N'$ has two dividing curves parallel to the longitude of $V_1.$ Moreover any two such solid tori with convex boundary each having two dividing curves of infinite slope have contactomorphic complements. 
\end{lemma}

We note that $N'$ from this lemma is the standard neighborhood of a Legendrian knot $L$ topologically isotopic to $K_1.$ Moreover one easily checks that 
\[
\tb_\Q(L)=-\frac {p'}p
\]
where $p'<p$ is defined as in the previous sections. 
The next possible larger value for $\tb_\Q$ is $-\frac {p'}p + 1>-\frac 1p$ which violates the Bennequin bound.
\begin{thm}\label{thm:1}
The maximum possible value for the rational Thurston-Bennequin invariant for a Legendrian knot isotopic to $K_1$ is $-\frac {p'}p$ and it is uniquely realized, up to contactomorphism isotopic to the identity. Moreover, any Legendrian knot isotopic to $K_1$ with non-maximal rational Thurston-Bennequin invariant destabilizes.
\end{thm}
\begin{proof}
The uniqueness follows from the last sentence in Lemma~\ref{startclass}. The first part of the same lemma also establishes the destabilization result as it is well known, see \cite{EtnyreHonda01b}, that if the standard neighborhood of a Legendrian knot is contained in the standard neighborhood of another Legendrian knots then the first is a stabilization of the second.
\end{proof}

To finish the classification of Legendrian knots in the knot type of $K_1$ we need to identify the rational rotation number of the Legendrian knot $L$ in the knot type of $K_1$ with maximal rational Thurston-Bennequin invariant. To this end notice that if we fix the neighborhood $N'$ from Lemma~\ref{startclass} as the standard neighborhood of the maximal rational Thurston-Bennequin invariant Legendrian knot $L$ then we can choose a non-zero section $s$ of $\xi|_{\partial N'}.$ This allows us to define a relative Euler class for $\xi|_{N'}$ and $\xi|_C$ where $C={\overline{L(p,q)-N'}}.$ One easily sees that the Euler class of $\xi|_{N'}$ vanishes and the Euler class $e(\xi)$ is determined by its restriction to the solid torus $C.$ In particular, $\xi|_C$ is determined by
\[
e(\xi)(D)=e(\xi|_C,s)(D)\mod p,
\] 
where $D$ is the meridional disk of $C$ and the generator of $2$--chains in $L(p,q).$ Thinking of $D$ as the rational Seifert surface for $L$ we can arrange the foliation near the boundary to be by Legendrian curves parallel to the boundary (see \cite[Figure 1]{Honda00a}). From this we see that we can take a Seifert cable $L_c$ of $L$ to be Legendrian and satisfy 
\[
\rot_\Q(L)=\frac 1p \, \rot(L_c).
\]
By taking the foliation on $\partial N'=\partial C$ to be so that $D\cap \partial C$ is a ruling curve we see that 
\[
\rot_\Q(L)=\frac 1p\, \rot(L_c)=\frac 1p\, e(\xi|_C,s)(D).
\]
By the classification of tight contact structures on solid tori \cite{Honda00a} we see the number $e(\xi|_C,s)(D)$ is always a subset of  $\{p'-1-2k:k=0,1,\ldots, p'-1\}$ and determined by the Euler class of $\xi.$ To give a more precise classification we need to know the range of possible values for the Euler class of tight $\xi$ on $L(p,q).$ This is in principal known, but difficult to state in general. We consider several cases in the next subsection. 

We clearly have the analog of Theorem~\ref{thm:1} for $-K_1.$ That is all Legendrian knots in the knot type $-K_1$ destabilize to the unique maximal representative $L$ with $\tb_\Q(L)=-\frac{p'}{p}$ and rotation number the negative of the rotation number for the maximal Legendrian representative of $K_1.$

\begin{proof}[Proof of Theorem~\ref{thm:mainratknot}]
Notice that if $q^2\equiv \pm1 \mod p$ we have a diffeomorphism $\psi:L(p,q)\to L(p,q)$ that exchanges the Heegaard tori and if $q=1$ or $p-1$ then this diffeomorphism is isotopic to the identity. Thus when $p\not=2$ and $q=1$ or $p-1$ we have competed the proof of Theorem~\ref{thm:mainratknot}. Note also that we always have the diffeomorphism $\psi':L(p,q)\to L(p,q)$ that preserves each of the Heegaard tori but acts by complex conjugation on each factor of each Heegaard torus (recall that the Heegaard tori are $V_i=S^1\times D^2$ where $S^1$ and $D^2$ are a unit circle and disk in the complex plane, respectively). If $p=2$ then this diffeomorphism is also isotopic to the identity. Thus finishing the proof of Theorem~\ref{thm:mainratknot} in this case. We are left to consider the case when $q\not=1$ or $p-1.$ In this case we can understand $K_0$ and $-K_0$ by reversing the roles of $V_0$ and $V_1.$ That is we consider using the gluing map 
\[
\phi^{-1}=\begin{pmatrix} q& p\\ q' & p' \end{pmatrix}
\]
 to glue $\partial V_0$ to $\partial V_1.$
\end{proof}

\subsection{Classification results}

To give some specific classification results we recall that for the lens space $L(p,1), p$ odd,  there is a unique tight contact structure for any given  Euler class not equal to the zero class in $H_2(L(p,q);\Z).$ From this, the fact that $p'=p-1$ in this case, and the discussion in the previous subsection we obtain the following theorem.
\begin{thm}
For $p$ odd and any integer $l\in \{p-2-2k:k=0,1,\ldots, p-2\}$ there is a unique tight contact structure $\xi_l$ on $L(p,1)$ with $e(\xi_l)(D)=l$ (here $D$ is again the 2-cell in the  CW-decomposition of $L(p,1)$ given in the last subsection). In this contact structure the knot types $K_1$ and $-K_1$ are weakly Legendrian simple and transversely simple. Moreover the rational Thurston-Bennequin invariants realized by Legendrian knots in the knot type $K_1$ are
\[
\{-\frac{p-1}{p}-k: k \text{ a non-positive integer}\}.
\]
The range for Legendrian knots in the knot type $-K_1$ is the same. The range of rotation numbers realized for a Legendrian knot in the knot type $K_1$ with rational Thurston-Bennequin invariant $-\frac{p-1}{p}-k$ is 
\[
\{\frac{l}{p} + k-2m: m=0,\ldots, k \}
\]
and for $-K_1$ the range is 
\[
\{\frac{-l}{p} + k-2m: m=0,\ldots, k \}.
\]
The range of possible rational self-linking numbers for transverse knots in the knot type $K_1$ is
\[
\{-\frac{p+l-1}{p}-k: k \text{ a non-positive integer}\}
\]
and in the knot type $-K_1$ is
\[
\{-\frac{p-l-1}{p}-k: k \text{ a non-positive integer}\}.
\]
\end{thm}
Results for other $L(p,q)$ can easily be written down after the range of Euler classes for the tight contact structures is determined. 

\section{Rationally null-homologous links and uniform Seifert surfaces}\label{sec:links}

Much of our previous discussion for rational knots also applies to links, but many of the statements are a bit more awkward (or even uncertain) if we do not restrict to certain kinds of rational Seifert surfaces.     

Let $L=K_1 \cup \dots \cup K_n$ be an oriented link of $\Z$--homological order $r>0$ in a $3$--manifold $M$ and denote a tubular neighborhood of $L$ by $N(L) = N(K_1) \cup \dots \cup N(K_n)$.  By $X(L)$ denote the link exterior $\overline{M \setminus N(L)}$.  Fix a framing for each $N(K_i)$.
Since $L$ has order $r$, there is an embedded $(r,s_i)$--curve on $\bdry N(K_i)$ for each $i$ that together bound an oriented surface $\Sigma^\circ$ in $X(L)$.  Radially coning $\bdry \Sigma^\circ \subset N(L)$ to $L$ gives a surface $\Sigma$ in $M$ whose interior is embedded and for which $\bdry \Sigma |_{K_i}$ wraps $r$ times around $K_i$.  By tubing if needed, we may take $\Sigma$ to be connected.  Such a surface $\Sigma$ will be called a \dfn{uniform rational Seifert surface} for $L$, and we say $L$ \dfn{$r$--bounds} $\Sigma$.

Notice that as $\Z$--homology chains, $\bdry \Sigma = rL = 0$.  Since as $1$--chains there may exist varying integers $r_i$ such that $r_1 K_1 + \dots + r_n K_n = 0$, the link $L$ may have other rational Seifert surfaces that are not uniform.  However, only for a uniform rational Seifert surface $\Sigma$ do we have that $\bdry(\frac{1}{r} \Sigma) = L$ as $\Q$--homology chains.  

With respect to uniform rational Seifert surfaces, the definition of rational linking number for rationally null-homologous links extends directly:  If $L$ is an oriented link that $r$--bounds $\Sigma$ and $L'$ is another oriented link, then
\[
lk_\Q(L,L') = \frac 1r \Sigma \cdot L'
\]
with respect to $[\Sigma]$.  If $L'$ is rationally null-homologous and $r'$--bounds $\Sigma'$, then this linking number is symmetric and independent of choice of $\Sigma$ and $\Sigma'$.

It now follows that the entire content of Sections~\ref{sec:basics} and \ref{sec:bennequin} extends in a straightforward manner to transverse/Legendrian links $L$ that $r$--bound a uniform rational Seifert surface $\Sigma$ in a contact manifold.   
The generalization of Theorem~\ref{thm:supportgen} is straightforward as well, but relies upon the generalized statements of Lemmas~\ref{lem:cable} and \ref{lem:slcable}.  Rather than record the general statements of these lemmas (which becomes cumbersome for arbitrary cables), we present them only for integral resolutions  of rationally null-homologous links with uniform rational Seifert surfaces.

\begin{lemma} \label{lem:intresolutionlink}
Let $L$ be a link in $M$ that $r$--bounds a uniform rational Seifert surface $\Sigma$ for $r >0$.  Choose a framing on each component $K_i$ of $L$, $i=1, \dots,n$ so that $\Sigma$ approaches $K_i$ as $(r,s_i)$--curves.  Let $L'$ be the link formed by replacing each $K_i$ by its $(r,q_i)$--cable where $q_i \neq s_i$.  If $L$ is a rationally fibered link with fiber $\Sigma$, then $L'$ is a (null-homologous) fibered link bounding a fiber $\Sigma'$ with
\[
\chi(\Sigma') = \chi(\Sigma) + (1-r) \sum_{i=1}^n |s_i - q_i|.
\]
Furthermore, assume $M$ is endowed with a contact structure $\xi$ and $L$ is a transverse link.  If the integral resolution $L'$ of $L$ is positive and transverse then
\[ 
sl(L', [\Sigma']) = r\, sl_\Q(K,[\Sigma]) + (-1+r)\sum_{i=1}^n |s_i - q_i|.
\]
\end{lemma}

\begin{proof}
The construction of $\Sigma'$ is done by attaching $|s_i -q_i|$ copies of meridional disks of $N(K_i)$ with $r\,|s_i-q_i|$ half twisted bands to $\Sigma$ for each $i$.  Now follow the proof of Lemma~\ref{lem:slcable}.
\end{proof}

\begin{thm}\label{thm:supporgenlink}
Let $L$ be a rationally null-homologous, fibered transverse link in a contact $3$--manifold $(M,\xi)$ such that $\xi$ is tight when restricted to the complement of $L$.  Suppose $L$ $r$--bounds the fibers of the fibration of $M-L$ and let $\Sigma$ be a fiber. 
Then  $r\, sl^\xi_\Q(L,\Sigma)=-\chi(\Sigma)$ if and only if  either $\xi$ agrees with the contact structure supported by the rational open book determined by $L$ and $\Sigma$ or is obtained from it by adding Giroux torsion along tori which are incompressible in the complement of $L.$
\end{thm}

\begin{proof}
Follow the proof of Theorem~\ref{thm:supportgen} using Lemma~\ref{lem:intresolutionlink} instead of Lemmas~\ref{lem:cable} and \ref{lem:slcable}.
\end{proof}


\def\cprime{$'$} \def\cprime{$'$}

\end{document}